\documentclass[11pt]{amsart}
\usepackage{amsmath, amsthm, amssymb}
\usepackage{epsfig}
\usepackage{rawfonts}
\usepackage{enumerate}
\usepackage{graphics}
\usepackage{tikz}

\usepackage{xspace}
     \usepackage{graphicx}

\theoremstyle{plain}

\theoremstyle{theorem}

\newtheorem{defn}{Definition}[section]
\newtheorem{prop}[defn]{Proposition}
\newtheorem{thm}[defn]{Theorem}
\newtheorem{lemma}[defn]{Lemma}

\newtheorem{coro}[defn]{Corollary}
\newtheorem{exa}[defn]{Example}
\newtheorem{rmk}[defn]{Remark}

\newtheorem{alg}[defn]{Algorithm}

\theoremstyle{remark}

\begin{document}

\title[Irreducible GNS and uniqueness of the Frobenius element]{Irreducible Generalized Numerical Semigroups and uniqueness of the Frobenius element}

\author{Carmelo Cisto, Gioia Failla, Chris Peterson, Rosanna Utano}

\address{Universit\'{a} di Messina,  Dipartimento di Scienze Matematiche e Informatiche, Scienze Fisiche e Scienze della Terra\\
Viale Ferdinando Stagno D'Alcontres 31\\
98166 Messina, Italy}
\email{carmelo.cisto@unime.it}

\address{Universit\'{a} Mediterranea di Reggio Calabria, DIIES\\
Via Graziella, Feo di Vito,  Reggio Calabria, Italy}
\email{gioia.failla@unirc.it}

\address{Colorado State University\\
Department of mathematics, Fort Collins, CO 80523 USA}
\email{peterson@math.colostate.edu}

\address{Universit\'{a} di Messina, Dipartimento di Scienze Matematiche e Informatiche, Scienze Fisiche e Scienze della Terra\\
Viale Ferdinando Stagno D'Alcontres 31\\
98166 Messina, Italy}
\email{rosanna.utano@unime.it}

\begin{abstract}
Let $\mathbb{N}^{d}$ be the $d$-dimensional monoid of non-negative integers. A generalized numerical semigroup is a submonoid $ S\subseteq \mathbb{N}^d$ such that $H(S)=\mathbb{N}^d \setminus S$ is a finite set. We introduce irreducible generalized numerical semigroups and characterize them in terms of the cardinality of a special subset of $H(S)$. In particular, we describe relaxed monomial orders on $\mathbb N^d$, define the Frobenius element of $S$ with respect to a given relaxed monomial order, and show that the Frobenius element of $S$ is independent of the order if the generalized numerical semigroup is irreducible.

\noindent {\textit{Keywords}:}{ Generalized Numerical Semigroups (GNS), Irreducible, Symmetric and Pseudo-symmetric GNS, Frobenius element

}

\end{abstract}


\maketitle

\section*{Introduction}
Let $\mathbb{N}$ be the set of non negative integers. A \emph{numerical semigroup} is a submonoid $S$ of $\mathbb{N}$ such that $\mathbb{N}\setminus S$ is a finite set. The elements of $H(S)=\mathbb{N}\setminus S$ are called the \emph{holes} of $S$ (or \emph{gaps}) and the largest element in $H(S)$ is referred to as the \emph{Frobenius element} of $S$, denoted by $F(S)$. The number $g=|H(S)|$ is called the \emph{genus} of $S$. It is well known that every numerical semigroup $S$ has a unique minimal set of generators $G(S)$. In other words, every element in $S$ is a linear combination of elements in $G(S)$ (with coefficients in $\mathbb{N}$) and no subset of $G(S)$ has this property. A set $\{a_{1},a_{2},\ldots,a_{n}\}\subset \mathbb N$ generates a numerical semigroup if and only if 1 is the greatest common divisor of $a_{1},a_{2},\ldots,a_{n}$. See \cite{R.G.} for an interesting overview of the subject.

In \cite{Prof}, a straightforward generalization of numerical semigroups is described for submonoids of $\mathbb{N}^{d}$. A monoid $S\subseteq\mathbb{N}^{d}$ is called a generalized numerical semigroup (GNS) if $H(S)=\mathbb{N}^{d}\setminus S$  is a finite set. As for numerical semigroups, the elements of $H(S)$ are called the \emph{holes} of $S$ and the cardinality $g=|H(S)|$ is called the \emph{genus} of $S$. In \cite{Analele, Prof} several ideas originating in numerical semigroups were extended to generalized numerical semigroups and several new tools were introduced in order to handle the differences. A crucial difference between numerical semigroups in $\mathbb N$ and generalized numerical semigroups in $\mathbb N^d$ is the definition of the multiplicity and the Frobenius element of $S$. There is a natural ordering of the elements of $\mathbb N$ that respects the monoid operation of addition. This ordering is used to define the Frobenius element of $S$. While there is a natural partial ordering on $\mathbb{N}^d$, there is not a canonical total order. In \cite{Prof} the authors introduce relaxed monomial orders,  a total order that is a more general version of monomial orders that respects the monoid operation of addition in $\mathbb{N}^d$. Given a relaxed monomial order $\prec$ on $\mathbb N^d$ the Frobenius element of $S \subset \mathbb{N}^d$, with respect to $\prec$, is the largest element of $H(S)$. In this paper we introduce and characterize the {\it irreducible} generalized numerical semigroups  in terms of the cardinality of a special subset of its set of holes. We prove that, as in the case of numerical semigroups, every generalized numerical semigroup can be expressed as an intersection of finitely many irreducible GNSs. Moreover, for an irreducible GNS the Frobenius element is unique with respect to any relaxed monomial order. The structure of the paper is the following: Section 1 deals with the set of pseudo-Frobenius elements of a GNS. In Section 2 some results of \cite{RGM1} are extended to irreducible generalized numerical semigroups. In Section 3 we study the decomposition of a generalized numerical semigroup as an intersection of a finite number of irreducible ones. In Section 4 we restrict our attention to investigating conditions for which the Frobenius element does not depend on the relaxed monomial order. Section 5 presents characterizations of symmetric and pseudo-symmetric generalized numerical semigroups in terms of the Frobenius element and the genus.

\section{Useful tools}

\begin{defn} \rm Let $S\subseteq \mathbb{N}^{d}$ be a GNS. We define in $\mathbb{Z}^{d}$ the following relation: 
$$  \textbf{a}\leq_{S} \textbf{b}\ \mbox{if and only if} \ \textbf{ b}- \textbf{a}\in S$$
It is easy to see that $\leq_{S}$ is a partial order (relation) in $\mathbb{Z}^{d}$.
\end{defn}

\begin{defn} \rm Let $S\subseteq \mathbb{N}^{d}$ be a GNS.
\begin{enumerate}
\item[a)] Let $PF(S)=\{\textbf{x}\in H(S)\mid \textbf{x}+\textbf{s}\in S,\ \mbox{for all}\ \textbf{s}\in S\setminus\{\textbf{0}\}\}$. We call $PF(S)$
the set of \emph{pseudo-Frobenius elements} of $S$. The cardinality of  $PF(S)$ is called the \emph{type} of $S$.
\item[b)] Let $Ap(S,\textbf{n})=\{\textbf{s}\in S\mid \textbf{s}-\textbf{n}\notin S\}$ where $\textbf{n}\in S\setminus\{\textbf{0}\} $. We call $Ap(S,\textbf{n})$ the \emph{Ap\'ery set} of $S$ with respect to $\textbf{n}$.
\end{enumerate}
 \end{defn}

\begin{prop} Let $S\subseteq \mathbb{N}^{d}$ be a GNS. Then $PF(S)$ is the set of maximal elements in $H(S)$ with respect to $\leq_{S}$. \label{<S}\end{prop}
\begin{proof}
We take $\textbf{x}$ maximal in $\mathbb{Z}^{d}\setminus S$ with respect to $\leq_{S}$. If there exists  $\textbf{s}\in S$ such that $\textbf{x}+\textbf{s}\notin S$ then  $\textbf{x}\leq_{S} \textbf{x}+\textbf{s}$ which contradicts the maximality of $\textbf{x}$.
Conversely let $\textbf{x}\in PF(S)$. If there exists $\textbf{y}\in \mathbb{Z}^{d}\setminus S$ such that $\textbf{y}-\textbf{x}=\textbf{s}\in S$ then $\textbf{x}+\textbf{s}\notin S$ which is again a contradiction.
\end{proof}

Since $H(S)$ is a finite set, the previous proposition implies that $PF(S)$ is nonempty. In the next proposition, we let $\rm{Maximals}_{\leq_{S}}Ap(S,\textbf{n})$ denote the set of maximal elements in $Ap(S,\textbf{n})$ with respect to the partial order $\leq_{S}$.\\

\begin{prop} Let $S\subseteq \mathbb{N}^{d}$ be a GNS and $\textbf{n}\in S$ a nonzero element. Then
$$PF(S)=\{\textbf{w}-\textbf{n}\mid \textbf{w}\in \mbox{Maximals}_{\leq_{S}}Ap(S,\textbf{n})\}$$ \end{prop}
\begin{proof}
Let $\textbf{x}\in PF(S)$, then $\textbf{x}\notin S$ and $\textbf{x}+\textbf{n}\in S$, that is $\textbf{x}+\textbf{n}\in Ap(S,\textbf{n})$. If we consider $\textbf{w}\in Ap(S,\textbf{n})$ with $\textbf{x}+\textbf{n}\leq_{S} \textbf{w}$ it follows that $\textbf{y}=\textbf{w}-\textbf{x}-\textbf{n}\in S$ and $\textbf{w}-\textbf{n}=\textbf{x}+\textbf{y}\notin S$. So, since $\textbf{x}\in PF(S)$ it follows that $\textbf{y}=\textbf{0}$ and that $\textbf{w}=\textbf{x}+\textbf{n}$.
Now if $\textbf{w}\in \mbox{Maximals}_{\leq_{S}}Ap(S,\textbf{n})$ then $\textbf{w}-\textbf{n}\notin S$. If $\textbf{s}\in S$ and if $\textbf{w}-\textbf{n}+\textbf{s}\notin S$ then it follows that $\textbf{w}+\textbf{s}\in Ap(S,\textbf{n})$ which contradicts the maximality of $\textbf{w}$. 
\end{proof}

The previous proposition is true for numerical semigroups. It could be considered surprising that it occurs also in a GNS, because the Ap\'ery set may be an infinite set in a GNS. Let us illustrate the situation with an example.

\begin{exa} \rm Let $S=\mathbb{N}^{2}\setminus \{(0,1),(0,3),(1,0),(1,1),(1,3),(2,1),(2,0),(3,0)\}$. One computes that 
\begin{itemize}
\item $PF(S)=\{(0,3),(1,0),(1,3),(2,0),(2,1),(3,0)\}$
\item $Ap(S,(0,2))=\{(0,0),(0,5),(1,2),(1,5),(2,2),(2,3),(3,1),\\
(3,2),(4,0),(4,1),(5,0),(5,1),(n,0),(n,1)\mid n\geq 6\}$. 
\item $\rm{Maximals}_{\leq_{S}}Ap(S,(0,2))=\{(0,5),(1,2),(1,5),(2,2),(2,3),(3,2)\}$
\end{itemize}
Applying the previous proposition, the elements of $PF(S)$ are:\\ $(0,5)-(0,2)=(0,3)$, $(1,2)-(0,2)=(1,0)$,  $(1,5)-(0,2)=(1,3)$, $(2,2)-(0,2)=(2,0)$, $(2,3)-(0,2)=(2,1)$ and $(3,2)-(0,2)=(3,0)$.\end{exa}

\section{Irreducible Generalized Numerical Semigroups}

In this section we extend to the setting of a GNS some results that are formulated for numerical semigroups (see \cite{RGM1} and \cite{R.G.}). Their proofs require different arguments than for numerical semigroups. Let $\textbf{x}\in \mathbb{N}^{d}$, we denote by $x^{(i)}$ the $i$-th component for every $i\in \{1,\ldots,n\}$. In the following we denote by $\leq$ the natural partial order on $\mathbb{N}^{d}$, that is if $\textbf{x},\textbf{y}\in \mathbb{N}^{d}$, $\textbf{x}\leq \textbf{y}$ if and only if $x^{(i)}\leq y^{(i)}$ for every $i=1,\ldots,d$.  

\begin{defn} \rm Let $S\subseteq \mathbb{N}^{d}$ be a GNS. We define:
$$EH(S)=\left\lbrace \textbf{x}\in H(S)\mid 2\textbf{x}\in S\  \mbox{and}\  \textbf{x}+\textbf{s}\in S \ \mbox{for each}\ \textbf{s}\in S\setminus\{\textbf{0}\}\right\rbrace$$
\noindent and we call its elements \emph{special gaps}.
\end{defn}

\begin{rmk} \rm It is clear that $EH(S)\subseteq PF(S)$, but equality is not true in general. For instance, if $S=\mathbb{N}^{2}\setminus \{(0,1),(1,0),(2,0)\}$, we have $PF(S)=H(S)$ and $EH(S)=\{(0,1),(2,0)\}$.\\
Moreover $EH(S)$ is nonempty since, if $\textbf{f}$ is a maximal element in $H(S)$ with respect to the natural partial order in $\mathbb{N}^{d}$ (a maximal element exists because $H(S)$ is finite), then $\textbf{f}\in EH(S)$.\end{rmk}

\begin{prop} Let $S\subseteq \mathbb{N}^{d}$ be a GNS and $\textbf{x}\in H(S)$. Then $S\cup \{\textbf{x}\}$ is a semigroup if and only if $\textbf{x}\in EH(S)$. \label{amFr}\end{prop}
\begin{proof}
It is an easy consequence of the definition of $EH(S)$.
\end{proof}

\begin{defn} \rm Let $S\subseteq \mathbb{N}^{d}$ be a GNS. We call $S$  \emph{irreducible} if it cannot be expressed as the intersection of two GNSs properly containing $S$.\end{defn}

Now we prove two characterizations of irreducible generalized numerical semigroups. Similar results are formulated in \cite{RGM1} in the case of numerical semigroups.

\begin{prop} Let $S\subseteq \mathbb{N}^{d}$ be a GNS. $S$ is irreducible if and only if $|EH(S)|=1$. \label{EH1}\end{prop}
\begin{proof}
$\Rightarrow)$ Let $S$ be an irreducible GNS and suppose that there exist $\textbf{x},\textbf{y}\in EH(S)$ such that $\textbf{x}\neq \textbf{y}$. Then $S\cup \{\textbf{x}\}$ and $S\cup \{\textbf{y}\}$ are distinct GNSs properly containing $S$ and $\left( S\cup \{\textbf{x}\}\right)\cap \left(S\cup \{\textbf{y}\}\right)=S$ contradicting the irreducibility of $S$.\\
$\Leftarrow)$ Assume that $|EH(S)|=1$. We suppose there exist two different generalized numerical semigroups $S_{1}, S_{2}$, such that $S_{1}\supsetneq S$, $S_{2}\supsetneq S$ and $S_{1}\cap S_{2}=S$. Let $\textbf{x},\textbf{y}$ be maximal elements in $S_{1}\setminus S$ and $S_{2}\setminus S$, respectively with respect to the natural partial order in $\mathbb{N}^{d}$ (note that $S_{1}\setminus S$ and $S_{2}\setminus S$ are finite sets and those such maximal elements do exist). Obviously $\textbf{x},\textbf{y}\in H(S)$. Now we prove $\textbf{x},\textbf{y}\in EH(S)$. Since $\textbf{x}$ is maximal in $S_{1}\setminus S$ with respect to $\leq$, we have $2\textbf{x}\notin S_{1}\setminus S$, since $2\textbf{x}>\textbf{x}$. If $\textbf{s}\in S\setminus \{\textbf{0}\}$, then $\textbf{x}+\textbf{s}>\textbf{x}$, so $\textbf{x}+\textbf{s}\notin S_{1}\setminus S$. Furthermore $\textbf{x}\in S_{1}$ and $\textbf{s}\in S\setminus\{\textbf{0}\}\subset S_{1}\setminus\{\textbf{0}\}$, therefore $\textbf{x}+\textbf{s}\in S_{1}$ so $\textbf{x}+\textbf{s}\in S$. Hence it follows that $\textbf{x}\in EH(S)$. In a similar way we can prove that $\textbf{y}\in EH(S)$. By hypothesis $|EH(S)|=1$, then $\textbf{x}=\textbf{y}$, that is $\textbf{x}\in S_{1}\setminus S$ and $\textbf{x}\in S_{2}\setminus S$, hence $\textbf{x}\in S_{1}\cap S_{2}=S$, a contradiction since $\textbf{x}\in H(S)$.
\end{proof}

\begin{prop}  Let $S\subseteq \mathbb{N}^{d}$ be a GNS. $S$ is irreducible if and only if there exists $\textbf{f}\in H(S)$ such that for every $\textbf{h}\in H(S)$ with $2\textbf{h}\neq \textbf{f}$ we have that $\textbf{f}-\textbf{h}\in S$. \label{fund}
\end{prop} 
\begin{proof}
$\Rightarrow)$ By Proposition~\ref{EH1} $EH(S)$ consists of one element. Let $EH(S)=\{\textbf{f}\}$. Let $\textbf{h}\in H(S)$ with $\textbf{h}\neq \textbf{f}$ and suppose that $2\textbf{h}\neq \textbf{f}$. Since $\textbf{h}\notin EH(S)$ we have two possibilities:\\
1) Suppose there exists $\textbf{s}_{1}\in S\setminus\{\mathbf{0}\}$ such that $\textbf{f}_{1}=\textbf{h}+\textbf{s}_{1} \notin S$, in particular $\textbf{f}_{1}-\textbf{h}\in S$. If $\textbf{f}_{1}=\textbf{f}$ the assertion follows. If $\textbf{f}_{1}\neq \textbf{f}$ then $\textbf{f}_{1}\notin EH(S)$. We show that in any case there exist $\textbf{s}_{2}\in S\setminus \{\textbf{0}\}$ and $\textbf{f}_{2}\notin S$ with $\textbf{f}_{2}>\textbf{f}_{1}$ such that $\textbf{f}_2=\textbf{h}+\textbf{s}_{2}$. Since $\textbf{f}_{1}\notin EH(S)$, if there exists $\textbf{t}\in S\setminus \{\textbf{0}\}$ such that $\textbf{f}_{1}+\textbf{t}\notin S$, we put $\textbf{f}_{2}=\textbf{f}_{1}+\textbf{t}=\textbf{h}+(\textbf{s}_{1}+\textbf{t})$, $\textbf{s}_{2}=\textbf{s}_{1}+\textbf{t}$ so $\textbf{f}_{2}>\textbf{f}_{1}$. Otherwise, if $\textbf{f}_{1}+\textbf{t}\in S$ for every $t\in S\setminus\{\textbf{0}\}$ we consider $\textbf{f}_{2}=2\textbf{f}_{1}\notin S$, so $\textbf{f}_{2}=\textbf{h}+(\textbf{h}+2\textbf{s}_{1})$ and $\textbf{s}_{2}=\textbf{h}+2\textbf{s}_{1}=\textbf{f}_{1}+\textbf{s}_{1}\in S$. Therefore we proved that there exist $\textbf{s}_{2}\in S\setminus \{\textbf{0}\}$ and $\textbf{f}_{2}\notin S$ with $\textbf{f}_{2}>\textbf{f}_{1}$  such that $\textbf{f}_2=\textbf{h}+\textbf{s}_{2}$. If $\textbf{f}_{2}=\textbf{f}$ the assertion follows, on the contrary with the same argument we obtain a sequence of elements $\textbf{f}_{i}\notin S$ with $\textbf{f}_{i}>\textbf{f}_{i-1}$ for every $i$ and $\textbf{f}_{i}=\textbf{h}+\textbf{s}_{i}$ and $\textbf{s}_{i}\in S\setminus \{\textbf{0}\}$. Since $H(S)$ is a finite set, there exists $k\in \mathbb{N}$ such that $\textbf{f}_{k}=\textbf{f}$, furthermore $\textbf{f}_{k}-\textbf{h}\in S$.\\
2) Suppose that $\textbf{h}+\textbf{s}\in S$ for every $\textbf{s}\in S\setminus \{\textbf{0}\}$ and $2\textbf{h}\notin S$. We will prove that we obtain a contradiction. Observe that for every $i\in \mathbb{N}$ we have $i\textbf{h}+\textbf{s}\in S$ for every $\textbf{s}\in S\setminus \{\textbf{0}\}$. Since $H(S)$ is finite there exists $k=\max \{i \in \mathbb{N}\mid i\textbf{h}\notin S\}$, in particular $k\textbf{h}\in EH(S)$, that is $k\textbf{h}=\textbf{f}$. Since $2\textbf{h}\neq \textbf{f}$ it is $k\geq 3$. Consider the element $\overline{\textbf{h}}=(k-1)\textbf{h}$, we have $\overline{\textbf{h}}+\textbf{s}\in S$ for every $\textbf{s}\in S\setminus \{\textbf{0}\}$ and $2\overline{\textbf{h}}=2(k-1)\textbf{h}\in S$ since $2(k-1)>k$, that is $\overline{\textbf{h}}\in EH(S)$. But this is a contradiction since $\overline{\textbf{h}}\neq \textbf{f}$.\\
$\Leftarrow)$ By hypothesis, $\textbf{f}$ is greater than every element in $H(S)$ with respect to $\leq_{S}$, except for the element $\textbf{h}\in H(S)$ such that $2\textbf{h}=\textbf{f}$, if it exists. By Proposition~\ref{<S} the possible elements in $PF(S)$ are $\textbf{f}$ and $\textbf{h}=\frac{\textbf{f}}{2}$. Furthermore $EH(S)\subseteq PF(S)$ and $\textbf{h}\notin EH(S)$, since $2\textbf{h}=\textbf{f}\notin S$, so it must be $EH(S)=\{\textbf{f}\}$, hence $S$ is irreducible.
\end{proof}

\begin{lemma} Let $S\subseteq \mathbb{N}^{d}$ be an irreducible GNS with $EH(S)=\{\textbf{f}\}$. Then one and only one of these conditions is satisfied:
\begin{enumerate}
\item $PF(S)=\{\textbf{f}\} $ if there exists a component of $\textbf{f}$ that is odd.
\item $PF(S)=\left\lbrace\textbf{f},\frac{\textbf{f}}{2}\right\rbrace$ if all the components of $\textbf{f}$ are even.
\end{enumerate} \label{lsim} \end{lemma}
\begin{proof}
If $\textbf{f}$ has an odd component then there does not exist $\textbf{h}\in H(S)$ such that $2\textbf{h}=\textbf{f}$ and, by Proposition~\ref{fund}, $\textbf{f}$ is the maximal element in $H(S)$ with respect to $\leq_{S}$, so $PF(S)=\{\textbf{f}\}$ from Proposition~\ref{<S}.\\
If all components of $\textbf{f}$ are even then $\frac{\textbf{f}}{2}\in \mathbb{N}^{d}$ and such  an element is in $H(S)$, since $\textbf{f}\in H(S)$. It is $\textbf{f}-\frac{\textbf{f}}{2}=\frac{\textbf{f}}{2}\notin S$, then $\textbf{f}$ and $\frac{\textbf{f}}{2}$ are not comparable with $\leq_{S}$. Furthermore, by Proposition~\ref{fund}, $\textbf{f}$ is greater than all elements in $H(S)$ different from $\frac{\textbf{f}}{2}$ with respect to $\leq_{S}$, so $\textbf{f}$ is a maximal element with respect to that order, that is $\textbf{f}\in PF(S)$. Moreover, $\frac{\textbf{f}}{2}$ is maximal in $H(S)$ with respect to $\leq_{S}$, because, on the contrary, there exists $\textbf{h}\in H(S)$ such that $\frac{\textbf{f}}{2}\leq_{S}\textbf{h}\leq_{S} \textbf{f}$, but this is a contradiction. We conclude that $PF(S)=\{\textbf{f},\frac{\textbf{f}}{2}\}$.
\end{proof}

\begin{rmk} \rm
Notice that if one of the conditions of Lemma~\ref{lsim} is satisfied then $S$ is irreducible. In fact, in both cases $PF(S)=EH(S)=\{\textbf{f}\}$ and the assertion follows from Proposition~\ref{EH1}.
\end{rmk}
\noindent We gather the previous results in the following theorems.

\begin{thm} Let $S\subseteq \mathbb{N}^{d}$ be a GNS. Then the following statements are equivalent:
\begin{enumerate}
\item $|PF(S)|=1$.
\item $PF(S)=\{\textbf{f}\}$ and $\textbf{f}$ has at least one component that is odd.
\item There exists $\textbf{f}\in H(S)$ such that, for all $\textbf{h}\in H(S)$, $\textbf{f}-\textbf{h}\in S$. 
\end{enumerate}
\label{sym}\end{thm}

\begin{thm} Let $S\subseteq \mathbb{N}^{d}$ be a GNS. Then the following statements are equivalent:
\begin{enumerate}
\item $PF(S)=\{\textbf{f}, \frac{\textbf{f}}{2}\}$.
\item There exists $\textbf{f}\in H(S)$ such that its components are all even and for all $\textbf{h}\in H(S)$ with $\textbf{h}\neq\frac{\textbf{f}}{2}$ we have $\textbf{f}-\textbf{h}\in S$. 
\end{enumerate}
\label{p-sym}\end{thm}

\begin{defn} \rm A GNS $S\subseteq \mathbb{N}^{d}$  is called \emph{symmetric} if it satisfies one of the equivalent statements of Therorem~\ref{sym}. $S$ is called \emph{pseudo-symmetric} if it satisfies one of the equivalent statements of Theorem~\ref{p-sym}. \end{defn}
By Lemma~\ref{lsim}, any irreducible GNS is either symmetric or pseudo-symmetric, as well as for numerical semigroups.
\noindent 

\begin{exa} \rm  Let $S=\mathbb{N}^{2}\setminus \{(0,1),(1,1),(2,1),(3,1),(4,1),(5,1),(6,1)\}$. From an elementary computation, we find that $PF(S)=\{(6,1)\}=EH(S)$, so $S$ is a symmetric GNS.\\
Let $S'=\mathbb{N}^{2}\setminus \{(1,0),(2,0),(3,0),(4,0),(5,0),(6,0),(12,0)\}$. In this case, we find that $PF(S)=\{(6,0),(12,0)\}$ and $EH(S)=\{(12,0)\}$, so $S'$ is pseudo-symmetric.\\
Both $S$ and $S'$ are irreducible GNSs.\end{exa}

\begin{rmk} \rm If $d=1$ and $S$ is a symmetric numerical semigroup, then $S$ is irreducible, so $EH(S)=\{f\}$ and by Theorem~\ref{sym} $f$ is odd. Moreover $f$ is the Frobenius number of the numerical semigroup $S$, because for every numerical semigroup the Frobenius number is trivially an element of $EH(S)$. Therefore, the definition of a symmetric GNS, provided here, is a generalization of the definition of a symmetric numerical semigroup. The same argument holds for a pseudo-symmetric GNS. \end{rmk}

\section{Decomposition of a GNS as an intersection of finitely many irreducible ones}

It is known that every numerical semigroup can be expressed as an intersection of a finite number of irreducible numerical semigroups. A decomposition with the least number of irreducible numerical semigroups involved can be obtained algorithmically (see \cite{RGM1}). The aim of this section is to obtain analogous results in our context.  

\begin{defn} \rm
Let $S\subseteq \mathbb{N}^{d}$ be a GNS. We define the sets:
\begin{enumerate}
\item $\mathcal{O}(S)=\{T\subseteq \mathbb{N}^{d}\mid T\ \mbox{is a GNS},\ S\subseteq T\}$, named the set of the \emph{oversemigroups} of $S$.
\item $\mathcal{I}(S)=\{T\in \mathcal{O}(S)\mid T\ \mbox{is irreducible}\}$.
\end{enumerate}

\end{defn}

\noindent Observe that $\mathcal{O}(S)$ is a finite set since $S$ has finite complement in $\mathbb{N}^{d}$, moreover $\mathcal{I}(S)\subseteq \mathcal{O}(S)$.\\
We know that if $S\subseteq\mathbb{N}^{d}$ is a GNS and $\textbf{x}\notin S$ then $S\cup \{\textbf{x}\}$ is a GNS if and only if $\textbf{x}$ is a special gap of $S$. In particular, in order to obtain the set of oversemigroups of $S$ it suffices to compute the set $EH(S)$, then for all $\textbf{x}\in EH(S)$ we compute $S_{\textbf{x}}=S\cup \{\textbf{x}\}$ and perform the procedure for all semigroups $S_{\textbf{x}}$ to obtain $\mathbb{N}^{d}$.

\begin{prop}
Let $S\subseteq \mathbb{N}^{d}$ be a GNS. If $S$ is not irreducible then $S=S_{1}\cap \cdots \cap S_{n}$, with $S_{1},\ldots,S_{n} \in \mathcal{I}(S)$.
\end{prop}

\begin{proof}
If $S$ is not irreducible then $S=S_{1}\cap S_{2}$, where $S_{1}$ and $S_{2}$ are GNSs such that $S\subset S_{1}$ and $S\subset S_{2}$. If $S_{1}$ and $S_{2}$ are not irreducible then we can repeat for them the same argument of $S$. Finally, we obtain $S$ as an intersection of a finitely many irreducible GNSs, since $\mathcal{I}(S)$ is a a finite set.
\end{proof}

\noindent Let $\mathrm{Minimals}_{\subseteq}\mathcal{I}(S)$ be the set of elements in $\mathcal{I}(S)$ that are minimal with respect to set inclusion. A decomposition  $S=S_{1}\cap \cdots \cap S_{n}$ of $S$, with $S_{i}\in \mathcal{I}(S)$ for every $i$, is called \emph{minimal} (or \emph{not refinable}) if  $S_{1},\ldots,S_{n} \in \mathrm{Minimals}_{\subseteq}\mathcal{I}(S)$.

\begin{prop} \label{minimalC}
Let $S\subseteq \mathbb{N}^{d}$ be a GNS and $S=S_{1}\cap \cdots \cap S_{n}$ with $S_{1},\ldots,S_{n} \in \mathcal{I}(S)$. Then there exist $S_{1}',\ldots,S_{n}' \in \mathrm{Minimals}_{\subseteq}\mathcal{I}(S)$ such that $S=S_{1}'\cap \cdots \cap S_{n}'$.
\end{prop}

\begin{proof}
If $S=S_{1}\cap \cdots \cap S_{n}$ and there exists $i\in \{1,\ldots,n\}$ such that $S_{i}\notin \mathrm{Minimals}_{\subseteq}\mathcal{I}(S)$ then we can choose $S_{i}'\subset S_{i}$ with $S_{i}'\in \mathrm{Minimals}_{\subseteq}\mathcal{I}(S)$.
\end{proof}

\begin{lemma} \label{TminusS}
Let $S$ and $T$ be two generalized numerical semigroups in $\mathbb{N}^{d}$ such that $S\subsetneq T$. Let $\textbf{h}\in \mathrm{Maximals}(T\setminus S)$ (maximal with respect to the natural partial order in $\mathbb{N}^{d}$). Then $\textbf{h}\in EH(S)$. 
\end{lemma}

\begin{proof}
Let us denote with $\leq$ the natural partial ordering in $\mathbb{N}^{d}$ and let  $\textbf{h}\in \mathrm{Maximals}(T\setminus S)$. Then $\textbf{h}\in H(S)$ and for all $\textbf{s}\in S\setminus\{\textbf{0}\}$ we have that $\textbf{h}+\textbf{s}\in T$ and $\textbf{h}+\textbf{s}>\textbf{h}$, so $\textbf{h}+\textbf{s}\in S$. Analogously $2\textbf{h}\in T$ and $2\textbf{h}>\textbf{h}$ so $2\textbf{h}\in S$. Thus $\textbf{h}\in EH(S)$.
\end{proof}

\begin{defn} \rm
Let $S\subseteq \mathbb{N}^{d}$ be a GNS and $T\in \mathcal{O}(S)$. We define:
$$\mathcal{C}(T)=\{\textbf{h}\in EH(S)\mid \textbf{h}\notin T\}$$

\end{defn}

\begin{prop}
Let $S\subseteq \mathbb{N}^{d}$ be a GNS and let $S_{1},\ldots,S_{n}\in \mathcal{O}(S)$. Then the following conditions are equivalent:
\begin{itemize}
\item[1.] $S=S_{1}\cap \cdots \cap S_{n}$
\item[2.] For all $\textbf{h}\in EH(S)$ there exists $i\in \{1,\ldots,n\}$ such that $\textbf{h}\notin S_{i}$.
\item[3.] $\mathcal{C}(S_{1})\cup \cdots \cup \mathcal{C}(S_{n})=EH(S)$.
\end{itemize}
\label{decompose}
\end{prop}

\begin{proof}
$1.\Rightarrow 2.$ Let $\textbf{h}\in EH(S)$. Then $\textbf{h}\notin S=S_{1}\cap \cdots \cap S_{n}$, that is $\textbf{h}\notin S_{i}$ for some $i\in \{1,\ldots,d\}$.\\
$2.\Rightarrow 1.$ Suppose $S\subsetneq S_{1}\cap \cdots \cap S_{n}$. Then let $\textbf{h}\in \mathrm{Maximals}((S_{1}\cap \cdots \cap S_{n})\setminus S)$, by Lemma~\ref{TminusS} we have $\textbf{h}\in EH(S)$ and $\textbf{h}\in S_{1}\cap \cdots \cap S_{n}$, that is a contradiction.\\
$2.\Leftrightarrow 3.$ It is trivial.
\end{proof}

\noindent As in the case of numerical semigroups it is possible to consider a minimal decomposition into irreducibles and to produce an algorithm to compute such a decomposition.

\begin{alg}\rm Let $S\subseteq \mathbb{N}^{d}$ be a not irreducible GNS.

\begin{enumerate}
\item Compute the set $EH(S)$.
\item Set $I=\emptyset$ and $C=\{S\}$.
\item For all $S'$ in $C$ let $B$ be the set of generalized numerical semigroups $\overline{S}$ such that $|\overline{S}\setminus S'|=1$.
\item Remove from $B$ the generalized numerical semigroups $S'$ such that $EH(S)\subseteq S'$.
\item Remove from $B$ the generalized numerical semigroups $S'$ such that there exists $T\in I$ with $T\subseteq S'$.
\item Set $C=\{S'\in B\mid S'\ \mbox{is not irreducible}\}$.
\item Set $I=\{S'\in B\mid S'\ \mbox{is irreducible}\}$.
\item If $C\neq \emptyset$ go to Step 3.
\item For every $S'\in I$, compute $\mathcal{C}(S')$.
\item Return a set of semigroups $S_{1}',\ldots,S_{r}'$ that are minimal in $I$ and 
$$\mathcal{C}(S_{1}')\cup \cdots \cup \mathcal{C}(S_{r}')=EH(S).$$
\end{enumerate}
\label{algirr}
\end{alg}

We explain briefly some lines of the previous algorithm:
\begin{itemize}
\item Step 3: The semigroups $\overline{S}$ are obtained as $S'\cup \{\textbf{x}\}$ with $\textbf{x}\in EH(S')$.
\item Step 4: If $EH(S)\subseteq S'$ by Proposition~\ref{decompose} $S'$ does not occur in a representation of $S$ as an intersection of GNSs.
\item Step 5: Since we want to compute a minimal decomposition of $S$ as an intersection of irreducible semigroups we do not need the oversemigroups of a computed irreducible GNS.
\item Step 8: By step 4 and step 5 it will occur that $C$ will be empty at a certain iteration.
\item Step 10: Since $\mathrm{Minimals}_{\subseteq}\mathcal{I}(S)\subseteq I$ and we can obtain a minimal decomposition as in Proposition~\ref{minimalC}. 
\end{itemize}

\noindent It is known that for numerical semigroups a minimal decomposition as defined in Proposition~\ref{minimalC} is not unique and  it is not always minimal with respect to the number of the semigroups that appear in the decomposition. The same occurs for GNSs. In Step 10 of Algorithm~\ref{algirr} we could produce also a decomposition of a GNS containing the minimum number of irreducible components.\\

\vspace{6pt}

\noindent The set of special gaps of a GNS allows to obtain some properties on maximality of a given GNS in the set of all GNSs.

\begin{prop}
Let $S\subseteq \mathbb{N}^{d}$ be a GNS and $\{\textbf{h}_{1},\ldots, \textbf{h}_{t}\}\subseteq H(S)$. Then the following are equivalent:
\begin{itemize}
\item[1.] $S$ is maximal with respect to inclusion among the GNSs \ $T$ such that $T\cap \ \{\textbf{h}_{1},\ldots, \textbf{h}_{t}\}=\emptyset$.
\item[2.] $EH(S)\subseteq \{\textbf{h}_{1},\ldots, \textbf{h}_{t}\}$.
\end{itemize}
\label{maximality}
\end{prop}
\begin{proof}
$1.\Rightarrow 2.$ Let $\textbf{h}\in H(S)$ and suppose that $\textbf{h}\notin \{\textbf{h}_{1},\ldots, \textbf{h}_{t}\}$, then $S\subsetneq S\cup \{\textbf{h}\}$ and $(S\cup \{\textbf{h}\})\cap \{\textbf{h}_{1},\ldots, \textbf{h}_{t}\}=\emptyset$, a contradiction.\\
$2.\Rightarrow 1.$ Let $T$ be a GNS such that $T\cap \{\textbf{h}_{1},\ldots, \textbf{h}_{t}\}=\emptyset$ and suppose that $S\subsetneq T$. Then consider $\textbf{h}\in \mathrm{Maximals}_{\leq}(T\setminus S)$. By Lemma~\ref{TminusS}, $\textbf{h}\in EH(S)$, but  $\textbf{h}\notin \{\textbf{h}_{1},\ldots, \textbf{h}_{t}\}$, since $T\cap \ \{\textbf{h}_{1},\ldots, \textbf{h}_{t}\}=\emptyset$, that is a contradiction.
\end{proof}

This means that if $S$ and $T$ are two GNSs such that $EH(S)=EH(T)$ but $H(S)\neq H(T)$ then $S\nsubseteq T$ and $T\nsubseteq S$.

\section{Uniqueness of the Frobenius element}

In $\mathbb{N}^{d}$ there is not a natural total order so it is not immediately clear how to define  the Frobenius element for a GNS (as it is for numerical semigroups). In \cite{Prof} this aim is reached by defining a \emph{relaxed monomial order}, whose main purpose is to allow the building of the \emph{GNSs tree} for each $d$, in a similar way to the case $d=1$. The Frobenius element of a GNS is uniquely determined with respect to the defined relaxed monomial order. We want to investigate conditions for which the  Frobenius element does not depend on the chosen relaxed monomial order. 

\begin{defn}\rm
A total order, $\prec$, on the elements of $\mathbb{N}^d$ is called a {\it relaxed monomial order} if it satisfies:
\begin{itemize}
\item[i)] If $\mathbf{v},\mathbf{w}\in \mathbb N^d$ and if $\mathbf{v}\prec \mathbf{w}$ then $\mathbf{v}\prec \mathbf{w}+\mathbf{u}$ for any $\mathbf{u}\in \mathbb N^d$.
\item[ii)] If $\mathbf{v}\in \mathbb N^d$ and $\mathbf{v}\neq \mathbf{0}$ then $\mathbf{0}\prec \mathbf{v}$.
\end{itemize}
\end{defn} 

The well known monomial orders defined on the monomials of a given polynomial ring induce in $\mathbb{N}^{d}$ total orders which can be called \emph{monomial orders in $\mathbb{N}^{d}$}, and such orders are in particular relaxed monomial orders (the converse is not true). It will be useful to consider that any monomial order on the elements of $\mathbb N^d$ can be defined in terms of dot products via an ordered collection of $d$ linearly independent {\it weight vectors} in $\mathbb R^d_{\geq 0}$ \cite{CLO07}. More precisely, if $\mathbf{w}_1, \mathbf{w}_2, \dots, \mathbf{w}_d$ are linearly independent vectors in $\mathbb R^d_{\geq 0}$ then one can define a monomial order $<_m$ on the elements of $\mathbb N^d$ by $$\mathbf{u}<_m\mathbf{v} \iff {\rm the\ smallest\ } i \ {\rm for\ which \ }\mathbf{w}_i\cdot \mathbf{u}\neq \mathbf{w}_i\cdot \mathbf{v} {\rm \ has \ }\mathbf{w}_i\cdot \mathbf{u} < \mathbf{w}_i\cdot \mathbf{v}.$$
Furthermore $<_{m}$ is also a relaxed monomial order.

\begin{exa} \rm Examples of relaxed monomial orders in $\mathbb{N}^{d}$ are:

\begin{itemize}

\item[1.] The {\it lexicographic order} determined setting the weight vectors $\mathbf{w}_{i}=\textbf{e}_{i}$, where $\textbf{e}_{1}=(1,0,0,\dots, 0), \textbf{e}_{2}=(0,1,0,\dots,0), \dots, \textbf{e}_{d}=(0,0,\dots, 1)$. This is also a monomial order.

\item[2.] Every permutation of the elements  $\textbf{e}_{i}$, for $i=1,\ldots,d$ determines a relaxed monomial order.

\item[3.] Let $<_{1}$ be a monomial order and and let $\min(\mathbf{u})=\min \{u^{(i)}\mid i=1,\ldots,d\}$. Define $\mathbf{u}\prec \mathbf{v}$ if, 
\vskip 1pt
i) $\min(\mathbf{u})< \min(\mathbf{v})$ or if
\vskip 1pt
ii) $\min(\mathbf{u})= \min(\mathbf{v})$ and $\mathbf{u}<_1\mathbf{v}$.\\
Note that $\prec$ is a relaxed monomial order but is not a monomial order in general.
\end{itemize} \label{example} \end{exa}

\begin{defn} \rm  Let $S\subseteq \mathbb{N}^{d}$ be a GNS and $\prec$ be a relaxed monomial order in $\mathbb{N}^{d}$. The \emph{Frobenius element} of $S$ with respect to $\prec$, denoted by $\textbf{F}_{\prec}$, is the largest element of $H(S)$ with respect to $\prec$. \end{defn}

\begin{prop} Every relaxed monomial order in $\mathbb{N}^{d}$ extends the natural partial order in $\mathbb{N}^{d}$.\label{relax}
\end{prop}
\begin{proof}
Let $\textbf{a},\textbf{b}\in \mathbb{N}^{d}$ be distinct elements with $\textbf{a}\leq \textbf{b}$, so there exists $\textbf{c}\in \mathbb{N}^{d}$ such that $\textbf{a}+\textbf{c}=\textbf{b}$. Furthermore, let $\prec$ be a relaxed monomial ordering in $\mathbb{N}^{d}$. Suppose that $\textbf{b}\prec \textbf{a}$, then $\textbf{b}\prec \textbf{a}+\textbf{c}=\textbf{b}$, a contradiction.
\end{proof}

\begin{prop} Let $S\subseteq \mathbb{N}^{d}$ be a GNS and $\textbf{f}\in H(S)$. Then $\textbf{F}_{\prec}=\textbf{f}$ for every relaxed monomial order $\prec$ if and only if $\textbf{f}$ is the unique maximal element in $H(S)$ with respect to the natural partial order in $\mathbb{N}^{d}$. \label{maxfr} 
\end{prop}
\begin{proof}
The sufficient condition is an easy consequence of Proposition~\ref{relax}. Observe that $\textbf{f}$ must be maximal in $H(S)$ with respect to the natural partial order in $\mathbb{N}^{d}$, because if there exists $\textbf{h}\in H(S)$ with $\textbf{f}\leq\textbf{h}$ then $\textbf{f}\preceq \textbf{h}$ for every relaxed monomial order in $\mathbb{N}^{d}$, by Proposition~\ref{relax}. We have to prove that $\textbf{f}$ is the unique maximal element. If there exists another maximal element $\textbf{g}\neq \textbf{f}$ then, since $\textbf{f}\nleq \textbf{g}$ and $\textbf{g}\nleq\textbf{f}$, $\textbf{g}$ has got at least one component, the $j$-th for instance, that is larger than the $j$-th component of $\textbf{f}$. We are going to define a relaxed monomial order, $\prec$, by assigning suitable weight vectors $\textbf{w}_{1},\textbf{w}_{2},\ldots,\textbf{w}_{d}$.  Let $\textbf{e}_{i}$, for $i=1,\ldots,d$, denote the standard basis vectors. Fix $\textbf{w}_{1}=\textbf{e}_{j},\textbf{w}_{j}=\textbf{e}_{1}$ and $\textbf{w}_{i}=\textbf{e}_{i}$ for $i\neq 1,j$. The relaxed monomial order defined in such a way leads to $\textbf{f}\prec \textbf{g}$, but this is a contradiction.
\end{proof}
By Proposition~\ref{maxfr}, if $S$ is a GNS such that there exists a unique maximal element $\textbf{f}\in H(S)$ with respect to the partial order $\leq$ in $\mathbb{N}^d$,  then $S$ has a unique Frobenius element, independent of the fixed relaxed monomial order. 
\begin{defn} \rm Let $S$ be a GNS such that there exists a unique maximal $\textbf{f}\in H(S)$ with respect to the partial order $\leq$ in $\mathbb{N}^d$. We call $(S,\textbf{f})$  a Frobenius GNS and we can refer to $\textbf{f}$ as the Frobenius element of $S$ without ambiguity. \end{defn}

\begin{rmk} \rm
 Every numerical semigroup is a Frobenius GNS $(S,f)$, where $f$ is the Frobenius number. 
 \end{rmk}

\begin{prop} Let $S\subseteq \mathbb{N}^{d}$ be an irreducible GNS with $EH(S)=\{\textbf{f}\}$. Then $(S,\textbf{f})$ is a Frobenius GNS.\label{frirr}\end{prop}
\begin{proof}
It suffices to prove that $\textbf{f}$ is the unique maximal element in $H(S)$ with respect to the natural partial order in $\mathbb{N}^{d}$. Let $\textbf{h}\in H(S)$, if $\textbf{h}=\frac{\textbf{f}}{2}$ then it is trivial that $\textbf{h}\leq \textbf{f}$. If $\textbf{h}\neq \textbf{f}$ then by Proposition~\ref{fund}, $\textbf{f}-\textbf{h}\in S\subseteq \mathbb{N}^{d}$, so $\textbf{h}\leq\textbf{f}$.
\end{proof}

\begin{rmk} \rm The converse of Proposition~\ref{frirr} is not true. Let $S=\mathbb{N}^{2}\setminus \{(1,0),(2,0),(3,0),(4,0),(6,0),(7,0),(9,0)\}$. It is easy to see that $(9,0)$ is the maximum in $H(S)$ with respect to natural partial order in $\mathbb{N}^{2}$, but $EH(S)=\{(9,0),(7,0),(6,0),(4,0)\}$. So $(S,(9,0))$ is a Frobenius GNS but it is not irreducible. \end{rmk}

\begin{coro} Symmetric and pseudo-symmetric GNSs are Frobenius GNSs. \end{coro}

\begin{prop}
Let $S\subseteq \mathbb{N}^{d}$ be a GNS. Then $S$ is irreducible with Frobenius element $\textbf{f}$ if and only if it is maximal in the set of GNSs not containing $\textbf{f}$.
\end{prop}
\begin{proof}
$(\Rightarrow)$ If $S$ is irreducible with Frobenius element $\textbf{f}$ then $EH(S)=\{\textbf{f}\}$ so it easily follows from Proposition~\ref{maximality}.\\
$(\Leftarrow)$ Suppose that $S=S_{1}\cap S_{2}$ with $S\subset S_{1}$ and $S\subset S_{2}$. Then $\textbf{f}\in S_{1}\cap S_{2}$, that is a contradiction.
\end{proof}

 Now we want investigate about the existence of a relaxed monomial order such that, with respect to it,  an element $\textbf{h}\in H(S)$ is the Frobenius element. By Proposition~\ref{relax} these elements must be maximal in $H(S)$ with respect to the natural partial order in $\mathbb{N}^{d}$.

\begin{defn} \rm Let $S\subseteq \mathbb{N}^{d}$ be a GNS and $\textbf{h}\in H(S)$. We call $\textbf{h}$ \emph{Frobenius allowable} if there exists a relaxed monomial order, $\prec$, such that $\textbf{F}_{\prec}=\textbf{h}$.\end{defn}

\begin{prop} Let $S\subseteq \mathbb{N}^{d}$ be a GNS whose hole set $H(S)$ has exactly two maximal elements, $\textbf{h}_{1},\textbf{h}_{2}$, with respect to the natural partial order in $\mathbb{N}^{d}$. Then both $\textbf{h}_{1}$ and $\textbf{h}_{2}$ are Frobenius allowable. \end{prop}
\begin{proof}
It suffices to prove that there exist relaxed monomial orders, $\prec_{1},\prec_{2}$, such that $\textbf{h}_{1}\prec_{1}\textbf{h}_{2}$ and $\textbf{h}_{2}\prec_{2}\textbf{h}_{1}$. Since $\textbf{h}_{1}$ and $\textbf{h}_{2}$ are distinct maximal elements then $\textbf{h}_{1}$ has  at least one coordinate, the $i$-th for instance, such that it is greater than the $i$-th coordinate of $\textbf{h}_{2}$, and $\textbf{h}_{2}$ has a coordinate, the $j$-th for instance, that it is greater than the $j$-th coordinate of $\textbf{h}_{1}$. We can define two relaxed monomial orders by weight vectors, as in the proof of Proposition~\ref{maxfr}. So we set $\prec_{1}$ by the assignements $\textbf{w}_{1}=\textbf{e}_{j}, \textbf{w}_{j}=\textbf{e}_{1}, \textbf{w}_{k}=\textbf{e}_{k}$ for $k\neq 1,j$, while for $\prec_{2}$ we choose $\textbf{w}_{1}=\textbf{e}_{i}, \textbf{w}_{i}=\textbf{e}_{1}, \textbf{w}_{k}=\textbf{e}_{k}$ for $k\neq 1,i$. In this way, $\textbf{h}_{1}\prec_{1}\textbf{h}_{2}$ and $\textbf{h}_{2}\prec_{2}\textbf{h}_{1} $ are both satisfied.
\end{proof}

\begin{exa} \rm If $H(S)$ has  more than two maximal elements, the same argument does not work, but  it can occur that any maximal element in $H(S)$ is Frobenius allowable, like in the following example.\\

Let $S=\mathbb{N}^{2}\setminus \{(0,1),(1,0),(1,1),(1,2),(1,3),(2,1),(3,0),(5,0),(7,0)\}$. The maximal elements in $H(S)$ with respect to the natural partial order in $\mathbb{N}^{2}$ are $(1,3),(2,1),(7,0)$. Let $\textbf{e}_{1}=(1,0),\textbf{e}_{2}=(0,1)$, then:
\begin{itemize}
\item $(1,3)$ is Frobenius allowable, in fact we can define $\prec$, relaxed monomial order by weight vectors $\textbf{w}_{1}=\textbf{e}_{2}$ and $\textbf{w}_{2}=\textbf{e}_{1}$. We have $(7,0)\prec (2,1) \prec (1,3) $.
\item $(7,0)$ is Frobenius allowable, in fact we can define $\prec$, relaxed monomial order by weight vectors $\textbf{w}_{1}=\textbf{e}_{1}$ and $\textbf{w}_{2}=\textbf{e}_{2}$. We have $(1,3)\prec (2,1) \prec (7,0) $.
\item $(2,1)$ is also Frobenius allowable with respect to the relaxed monomial order $\prec$ defined in Example ~\ref{example}, (3.),  where  $<_{1}$ is the lexicographic order. We have in  fact  $(7,0)\prec (1,3) \prec (2,1) $.
\end{itemize}
\end{exa}

\section{Formulas for Irreducible and Frobenius GNSs}

In this section we give further characterizations of symmetric and pseudo-symmetric GNSs in terms of the Frobenius element and the number of holes. Let us start by giving some notations. 

\begin{defn}
	\rm Let $S\subseteq \mathbb{N}^{d}$ be a GNS, $\textbf{h}\in \mathbb{N}^{d}$, and $\leq$ the natural partial order in $\mathbb{N}^{d}$. We define the sets:
	\begin{itemize}
		\item $\pi(\textbf{h})=\{\textbf{n}\in \mathbb{N}^{d}\mid \textbf{n}\leq \textbf{h}\}$.
		\item $LH(\textbf{h})=\{\textbf{g}\in H(S)\mid \textbf{g}\leq \textbf{h}\}$.
		\item $N(\textbf{h})=\{\textbf{n}\in \pi(\textbf{h})\mid \textbf{n}\in S\}$.
		\item $MH(S)$ the set of maximal elements in $H(S)$, with respect to $\leq$.
	\end{itemize}
\end{defn}

\begin{lemma} Let $\textbf{h}\in \mathbb{N}^{d}$ and let $S\subseteq \mathbb{N}^{d}$ be a GNS. Then the following hold:
	\begin{enumerate}
		\item $|\pi(\textbf{h})|=(h^{(1)}+1)(h^{(2)}+1)\cdots (h^{(d)}+1)$.
		\item $|N(\textbf{h})|+ |LH(\textbf{h})|=(h^{(1)}+1)(h^{(2)}+1)\cdots (h^{(d)}+1)$.
	\end{enumerate}
	where $|A|$ denotes the cardinality of the set $A$. \label{card} \end{lemma}
\begin{proof} Observe that $\pi(\textbf{h})$ represents the set of integer points of the hyper-rectangle whose vertices are $\textbf{h}$, the origin of axes, and the points in the coordinate planes $(h^{(1)},0,\ldots,0),(0,h^{(2)},0,\ldots,0),\ldots,(0,\ldots,0,h^{(d)})$. See the figure below: the marked points are the elements of the set $\pi((5,4))$ in $\mathbb{N}^{2}$.

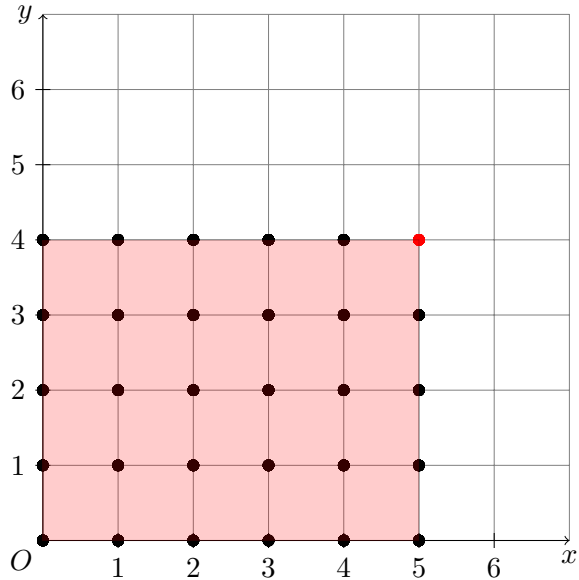
\begin{tikzpicture} 
\draw [help lines] (0,0) grid (7,7);
\draw [<->] (0,7) node [left] {$y$} -- (0,0)
-- (7,0) node [below] {$x$};
\foreach \i in {1,...,6}
\draw (\i,1mm) -- (\i,-1mm) node [below] {$\i$} 
(1mm,\i) -- (-1mm,\i) node [left] {$\i$}; 
\node [below left] at (0,0) {$O$};
\draw [mark=*] plot (0,0);
\draw [mark=*] plot (1,0);
\draw [mark=*] plot (2,0);
\draw [mark=*] plot (3,0);
\draw [mark=*] plot (4,0);
\draw [mark=*] plot (5,0);
\draw [mark=*] plot (0,1);
\draw [mark=*] plot (1,1);
\draw [mark=*] plot (2,1);
\draw [mark=*] plot (3,1);
\draw [mark=*] plot (4,1);
\draw [mark=*] plot (5,1);
\draw [mark=*] plot (0,2);
\draw [mark=*] plot (1,2);
\draw [mark=*] plot (2,2);
\draw [mark=*] plot (3,2);
\draw [mark=*] plot (4,2);
\draw [mark=*] plot (5,2);
\draw [mark=*] plot (0,3);
\draw [mark=*] plot (1,3);
\draw [mark=*] plot (2,3);
\draw [mark=*] plot (3,3);
\draw [mark=*] plot (4,3);
\draw [mark=*] plot (5,3);
\draw [mark=*] plot (0,3);
\draw [mark=*] plot (0,4);
\draw [mark=*] plot (1,4);
\draw [mark=*] plot (2,4);
\draw [mark=*] plot (3,4);
\draw [mark=*] plot (4,4);
\draw [red, mark=*] plot (5,4);
\draw [fill=red, opacity=0.2] (0,0) rectangle (5,4);
\end{tikzpicture}

\noindent So it is easy to deduce the first assertion.\\
The second follows from the remark that  $\pi(\textbf{h})=N(\textbf{h})\cup LH(\textbf{h})$ for all $\textbf{h}\in \mathbb{N}^{d}$, moreover $N(\textbf{h})$ and $LH(\textbf{h})$ are disjoint. \end{proof}

\begin{prop} Let $(S,\textbf{f})$ be a Frobenius GNS of genus $g$, in $\mathbb{N}^{d}$. Then $2g\geq (f^{(1)}+1)\cdots(f^{(d)}+1)$. \label{genus}\end{prop}
\begin{proof}
	In this case one can verify that $g=|H(S)|=|LH(\textbf{f})|$, moreover $|LH(\textbf{f})|\geq |N(\textbf{f})|$. Therefore $(f^{(1)}+1)\cdots(f^{(d)}+1)=|LH(\textbf{f})|+|N(\textbf{f})|\leq 2g$.  
\end{proof}

\noindent Every numerical semigroup is a Frobenius GNS and the previous proposition provides the well known inequality $g\geq \frac{F(S)+1}{2}$, where $F(S)$ and $g$ are respectively the Frobenius number and the genus of the given numerical semigroup.

\begin{defn}\rm Let $S\subseteq \mathbb{N}^{d}$ be a GNS and $\textbf{h}\in H(S)$. Then we define the following map:
	$$\Psi_{\textbf{h}}:N(\textbf{h})\rightarrow LH(\textbf{h}),\ \textbf{s}\longmapsto \textbf{h}-\textbf{s}$$
	It is easy to see that the map is well defined and it is injective.
	\label{map}\end{defn}

\begin{lemma} Let $S\subseteq \mathbb{N}^{d}$ be a GNS of genus $g$ and $\textbf{h}\in H(S)$. Then $|N(\textbf{h})|\leq |LH(\textbf{h})|\leq |H(S)|=g$.\label{dis} \end{lemma}
\begin{proof} It follows easily since the map $\Psi_{\textbf{h}}$ is injective. \end{proof}


Now we provide new characterizations for symmetric and pseudo-symmetric GNSs.

\begin{thm} Let $S\subseteq \mathbb{N}^{d}$ be a GNS of genus $g$. Then $S$ is symmetric if and only if there exists $\textbf{f}\in H(S)$ with $2g=(f^{(1)}+1)(f^{(2)}+1)\cdots (f^{(d)}+1)$. Furthermore $\textbf{f}$ is the Frobenius element of $S$. \label{rel}\end{thm}
\begin{proof}
	($\Rightarrow$) We suppose $S$ is symmetric. Then $EH(S)=PF(S)=\{\textbf{f}\}$, so $LH(\textbf{f})=H(S)$ by Proposition~\ref{frirr}. Let us prove that the map $\Psi_{\textbf{f}}$ is bijective, for this it suffices to prove that is surjective. If $\textbf{h}\in LH(\textbf{f})$, since $S$ is symmetric, then $\textbf{s}=\textbf{f}-\textbf{h}\in S$, therefore $\Psi_{\textbf{f}}(\textbf{s})=\textbf{h}$ so the map is surjective. It follows that $|N(\textbf{f})|=|LH(\textbf{f})|=g$ and $2g=|N(\textbf{f})|+|LH(\textbf{f})|=(f^{(1)}+1)(f^{(2)}+1)\cdots (f^{(d)}+1)$ (Proposition~\ref{card}).\\
	($\Leftarrow$) Let $\textbf{f}\in H(S)$ be such that $2g=(f^{(1)}+1)(f^{(2)}+1)\cdots (f^{(d)}+1)$. From Lemma~\ref{dis} and Proposition~\ref{card} it follows that $2g=|N(\textbf{f})|+|LH(\textbf{f})|\leq 2|LH(\textbf{f})|\leq 2g$. So $|LH(\textbf{f})|=g$ and also $|N(\textbf{f})|=g$, hence the map $\Psi_{\textbf{f}}$ is bijective. Now we prove that for every $\textbf{h}\in H(S)$ we have $\textbf{f}-\textbf{h}\in S$. Since $|LH(\textbf{f})|=g$ and $\Psi_{\textbf{f}}$ is surjective, then $LH(\textbf{f})=H(S)$ and if $\textbf{h}\in H(S)$ there exists $\textbf{s}\in S$ such that $\Psi_{\textbf{f}}(\textbf{s})=\textbf{f}-\textbf{s}=\textbf{h}$, in other words $\textbf{f}-\textbf{h}=\textbf{s}\in S$. From Theorem~\ref{sym} it follows that $S$ is symmetric, in particular $\textbf{f}\in MH(S)$, so it is the Frobenius element.
\end{proof}

\begin{thm} Let $S\subseteq \mathbb{N}^{d}$ be a GNS of genus $g$. Then $S$ is pseudo-symmetric if and only if there exists $\textbf{f}\in H(S)$ with $2g-1=(f^{(1)}+1)(f^{(2)}+1)\cdots (f^{(d)}+1)$. Furthermore $\textbf{f}$ is the Frobenius element of $S$ \label{p-rel}\end{thm}
\begin{proof}
	($\Rightarrow$) We suppose $S$ is pseudo-symmetric, so $PF(S)=\{\textbf{f},\frac{f}{2}\}$, $EH(S)=\{\textbf{f}\}$ and $LH(\textbf{f})=H(S)$. Moreover, for all $\textbf{h}\in H(S)$ with $\textbf{h}\neq\frac{\textbf{f}}{2}$ we have $\textbf{f}-\textbf{h}\in S$, so arguing as in the proof of Theorem~\ref{rel} we can prove that $|N(\textbf{f})|=|LH(\textbf{f})\setminus \{\frac{\textbf{f}}{2}\}|=g-1$. It follows that $(f^{(1)}+1)(f^{(2)}+1)\cdots (f^{(d)}+1)=|N(\textbf{f})|+|LH(\textbf{f})|=g+g-1=2g-1$.\\
	($\Leftarrow$) Let $\textbf{f}\in H(S)$ be such that $2g-1=(f^{(1)}+1)(f^{(2)}+1)\cdots (f^{(d)}+1)$, in particular every component of $\textbf{f}$ is an even number and $2g-1=|N(\textbf{f})|+|LH(\textbf{f})|\leq 2|LH(\textbf{f})|\leq 2g$. Therefore $|LH(\textbf{f})|=g$ (it is impossible that $2g-1=2|LH(\textbf{f})|$) and, as a consequence, $|N(\textbf{f})|=g-1$. Furthermore $\frac{\textbf{f}}{2}\in H(S)$ because $\textbf{f}\in H(S)$, so the map $\overline{\Psi}_{\textbf{f}}:N(\textbf{f})\rightarrow LH(\textbf{f})\setminus \{\frac{\textbf{f}}{2}\}$, induced by $\Psi_{\textbf{f}}$, is bijective. This implies that : for all $\textbf{h}\in LH(\textbf{f})\setminus \{\frac{\textbf{f}}{2}\}$, in other words $\textbf{h}\in H(S)$ and $\textbf{h}\neq \frac{\textbf{f}}{2}$, there exists $\textbf{s}\in S$ such that $\textbf{f}-\textbf{s}=\textbf{h}$, that is $\textbf{f}-\textbf{h}\in S$. Hence $S$ is pseudo-symmetric by Theorem~\ref{p-sym}, in particular $\textbf{f}\in MH(S)$ and it is the Frobenius element.
\end{proof}

\begin{exa}\rm Let $S=\mathbb{N}^{2}\setminus \{(0,1),(1,1),(2,1),(3,1),(4,1),(5,1),(6,1)\}$. $S$ is a GNS of genus $g=7$ and for $\textbf{f}=(6,1)\in H(S)$ the equality $2g=(6+1)(1+1)$ holds, so $S$ is symmetric. Indeed, we have $PF(S)=\{(6,1)\}$.\\
	Let $S'=\mathbb{N}^{2}\setminus \{(1,0),(2,0),(3,0),(4,0),(5,0),(6,0),(12,0)\}$. $S'$ is a GNS of genus $g=7$ and for the element $(12,0)$, it holds $2g-1=(12+1)(0+1)$, so $S'$ is pseudo-symmetric. Indeed, we have  $PF(S')=\{(12,0),(6,0)\}$.\end{exa}

\begin{rmk}\rm If $d=1$ then Theorem~\ref{rel} becomes: $S$ is symmetric if and only if $2g=F(S)+1$, where $F(S)$ is the Frobenius number of $S$, that is a well known result for numerical semigroups. From Theorem~\ref{p-rel} the corresponding result about pseudo-symmetric numerical semigroups follows.  \end{rmk}


\begin{exa}\rm
	Let $S=\mathbb{N}^{3}\setminus \{(1,0,0),(1,0,1),(2,0,0),(2,0,1)\}$. $S$ is a Frobenius GNS with Frobenius element $\textbf{f}=(2,0,1)$. Indeed $2g=8>(2+1)(0+1)(1+1)=6$.\\
	Let $S=\mathbb{N}^{3}\setminus \{(1,0,0),(1,1,0),(3,0,0),(3,1,0)\}$. In this case the Frobenius element is $\textbf{f}=(3,1,0)$ and $2g=8=(3+1)(1+1)(0+1)$, in particular $S$ is symmetric.
\end{exa}

Some open questions:
\begin{itemize} 
\item Are there other properties, as for numerical semigroups, satisfied by symmetric and pseudo-symmetric GNSs? 
\item Which are the  Frobenius allowable elements in a GNS? Are they all the maximal elements in $H(S)$ with respect to the natural partial order in $\mathbb{N}^{d}$ or only some of them?
\item Apart from numerical semigroups and irreducible GNSs, is it possible to classify other classes of Frobenius GNSs in which there is only one Frobenius element with respect to any relaxed monomial order?
\end{itemize}

\end{document}